\documentclass{amsart}
\usepackage{amssymb,amsmath,color}
\usepackage{color}
\usepackage{tikz}
\usepackage{tkz-graph}
\usetikzlibrary{calc, positioning}
\usepackage{pgf}
\usepackage{bm}
\usepackage{pgfplots}
 \usepackage[left=1in,right=1in,top=0.9in,bottom=0.9in]{geometry}

\def \hom {\operatorname{hom}}

\newcommand{\RR}{\mathbb{R}}

\newcommand{\NN}{\mathbb{N}}

\newcommand{\U}{\mathcal{U}}
\newcommand{\xx}{\mathbf{x}}
\newcommand{\yy}{\mathbf{y}}

\newtheorem{theorem}{Theorem}[section]

\newtheorem{lemma}[theorem]{Lemma}
\newtheorem{corollary}[theorem]{Corollary}

\theoremstyle{definition}
\newtheorem{definition}[theorem]{Definition}

\theoremstyle{remark}
\newtheorem{remark}[theorem]{Remark}

\title{Ubiquity of power sums in graph profiles}
\thanks{Grigoriy Blekherman was partially supported by NSF grant DMS-1901950. Annie Raymond was partially supported by NSF grant DMS-2054404.}

\author{Grigoriy Blekherman}
\address{School of Mathematics, Georgia Institute of Technology,
686 Cherry Street
Atlanta, GA 30332}\email{greg@math.gatech.edu}

\author{Annie Raymond}
\address{Department of Mathematics and Statistics,
Lederle Graduate Research Tower, 1623D,
University of Massachusetts Amherst
710 N. Pleasant Street
Amherst, MA 01003} \email{raymond@math.umass.edu}

\begin{document}

\maketitle

\begin{abstract}
Graph density profiles are fundamental objects in extremal combinatorics. Very few profiles are fully known, and all are two-dimensional. We show that even in high dimensions \emph{ratios} of graph densities and numbers often form the power-sum profile (the limit of the image of the power-sum map) studied recently 
by Acevedo, Blekherman, Debus and Riener. Our choice of graphs is motivated by recent work by Blekherman, Raymond and Wei on undecidability of polynomial inequalities in graph densities. While the ratios do not determine the complete density profile, they contain high-dimensional information. For instance, to reconstruct the density profile of $4k$-cycles from our results, one needs to solve only one-parameter extremal problems, for any number of $4k$-cycles.
\end{abstract}

\section{Introduction}The \emph{number of homomorphisms} from a graph $H$ to a graph $G$, denoted by $\hom(H;G)$, is the number of maps from $V(H)$ to $V(G)$ that yield a graph homomorphism, i.e., that map every edge of $H$ to an edge of $G$. The \emph{homomorphism density} from a graph $H$ to a graph $G$, denoted as $t(H;G):=\frac{\hom(H;G)}{|V(G)|^{|V(H)|}}$, is the probability that a random map from $V(H)$ to $V(G)$ yields a graph homomorphism. 
Many important problems and results in extremal graph theory can be framed as certifying the validity of polynomial inequalities in homomorphism numbers or densities which are valid on all graphs. For example, the Goodman bound \cite{goodman} (which implies Mantel's theorem \cite{mantel}) states that $\hom(K_1;G)\hom(K_3;G) \geq 2\hom(K_2;G)^2-\hom(K_2;G)\hom(K_1;G)^2$,
and Sidorenko's conjecture \cite{Sid93} can be stated as
$t(H;G) \geq  t(K_2;G)^{|E(H)|}$ for any bipartite graph $H$.

Understanding all $s$-tuples that can occur as either homomorphism numbers or densities for a fixed collection of graphs $\U=\{H_1,\dots,H_s\}$ is an extremely complicated problem. It is essentially equivalent to understanding all polynomial inequalities in homomorphism densities or numbers of the graphs $H_i$ which are valid on all target graphs $G$. It is known that the problem of checking whether a polynomial expression in either numbers or densities is nonnegative on all graphs is undecidable \cite{IR, HN11}.

 We call the set of all $s$-tuples the \emph{number} (resp. \emph{density}) \emph{(graph) profile} of the collection $\U$. To the best of our knowledge, full descriptions of all $s$-tuples are only known for pairs of graphs, and even then in a very limited number of cases; no higher-dimensional profiles are known. For instance an important result of Razborov \cite{RazTriangle} completely describes the density profile of $\mathcal{U}=\{K_2,K_3\}$. This was extended by Nikiforov to $\U=\{K_2,K_4\}$ in \cite{Nikiforov}, and generalized by Reiher to $\U=\{K_2,K_n\}$ in \cite{Reiher}.

In this paper, instead of computing number or density profiles, we compute some higher-dimensional profiles of \emph{ratios} of densities and numbers. We show that these ratio profiles are direct products of the \emph{power-sum profile} which was studied in \cite{abdr} under the name limit Vandermonde cell.

\begin{definition}
For $i\in \NN$ and $\xx\in \RR^n$, we let $p_i(\xx)=\sum_{j=1}^n x_j^i$ be the \emph{$i$-th power sum polynomial}. Let $\Pi_{n,\ell}:=\{(p_2(\xx), p_3(\xx), \ldots, p_\ell(\xx)) \mid \xx\in \RR^n_{\geq 0} \textup{ and } p_1(\xx)=1\}$. Let $\Pi_\ell:=\textup{cl}\left(\bigcup_{n\in \NN} \Pi_{n,\ell}\right)$. We call $\Pi_\ell$ the \emph{power-sum profile}.  
\end{definition}

Now we describe in detail the ratio profiles that we characterize. Let $N_{c,q}$ denote the cycle of length $c$ where every edge is replaced by a clique of size $q$. In particular, $N_{c,2}$ is simply the cycle of length $c$, denoted as $C_c$.  It was shown in \cite{BRW} that testing the validity of polynomial inequalities in densities of graphs $N_{4i,q}$ is undecidable. The proof made use of some geometric properties of the profiles of ratios of densities, which we now explore fully.

In Section~\ref{sec:necklaces}, we compute the density ratio profile $\mathcal{N}_{4\ell, \leq r}$ recording the closure of the points \small{$$\left(\frac{t(N_{8,2};G)}{(t(N_{4,2};G))^2}, \frac{t(N_{12,2};G)}{(t(N_{4,2};G))^3}, \ldots, \frac{t(N_{4\ell,2};G)}{(t(N_{4,2};G))^\ell}, \ldots, \frac{t(N_{8,r};G)}{(t(N_{4,r};G))^2}, \frac{t(N_{12,r};G)}{(t(N_{4,r};G))^3}, \ldots, \frac{t(N_{4\ell,r};G)}{(t(N_{4,r};G))^\ell}\right)$$}\normalsize{for} graphs $G$ such that $t(N_{4,q};G)\neq 0$ for $2\leq q\leq r$. We show that this profile is equal to  $\Pi_\ell^{r-1}$, the direct product of $r-1$ copies of the power-sum profile $\Pi_\ell$.

Thus we also recover the density ratio profile $$\mathcal{C}_{4\ell}:=\textup{cl}\left(\left\{\left(\frac{t(C_8;G)}{(t(C_4;G))^2}, \frac{t(C_{12};G)}{(t(C_4;G))^3}, \ldots, \frac{t(C_{4\ell};G)}{(t(C_4;G))^\ell}\right) \mid G \textup{ is a graph s.t.\ } t(C_{4};G)\neq 0\right\}\right)$$ which is equal to $\Pi_\ell$. 
This allows us to understand the (usual) density profile $$\mathcal{C}^*_{4\ell}:=\textup{cl}\left(\left\{\left(t(C_4,G), t(C_8;G), \ldots, t(C_{4\ell};G)\right) \mid G \textup{ is a graph}\right\}\right)$$ up to one dimension because we divided out by $t(C_4;G)$. To obtain the actual profile $\mathcal{C}^*_{4\ell}$, we need to compute the fiber (preimage) over any point $\mathbf{a}=(a_2,\dots, a_\ell)\in\Pi_\ell$ which can be done by finding $$w_{\mathbf{a}}:=\max_G \{t(C_4;G) \mid t(C_{4i};G)=a_i \ \forall 2 \leq i \leq \ell\},$$ i.e., the maximal $C_4$-density that achieves the given point $\mathbf{a}$. Then the fiber over $\mathbf{a}$ consists of all points $(t, a_2t^2,a_3t^3,\dots, a_\ell t^{\ell}) $ with $0\leq t\leq w_{\mathbf{a}}$. Indeed, by our definition of $w_{\mathbf{a}}$, $t=w_{\mathbf{a}}$ is realizable (as a limit) by some sequence of graphs $G_n$ increasing in size. Taking the disjoint union of $G_n$ with some disconnected vertices, we can realize all points with smaller density $t$. The same ideas carry on exactly to the density ratio profile $\mathcal{N}_{4\ell, q}$ defined as $$\textup{cl}\left(\left\{\left(\frac{t(N_{8,q};G)}{(t(N_{4,q};G))^2}, \frac{t(N_{12,q};G)}{(t(N_{4,q};G))^3}, \ldots, \frac{t(N_{4\ell,q};G)}{(t(N_{4,q};G))^\ell}\right) \mid G \textup{ is a graph s.t.\ } t(N_{4,q};G)\neq 0\right\}\right)$$ for some fixed $q\geq 2$. For the most general density ratio profile $\mathcal{N}_{4\ell,\leq r}$, since we divide by $r-1$ separate quantities, we recover the (usual) density profile  $$\mathcal{N}^*_{4\ell,\leq r}:=\textup{cl}\left(\left\{\left(t(N_{4,2};G), t(N_{8,2};G), \ldots,(N_{4\ell,2};G), \ldots t(N_{4,r};G), t(N_{8,r};G), \ldots,(N_{4\ell,r};G)\right) \mid G \textup{ is a graph}\right\}\right)$$ up to $r-1$ dimensions; see Remark~\ref{rem:lossofdimensionsgeneral} for details. Note that we are losing a fixed number of dimensions even though the dimension of the profile can get arbitrarily large. 

 In Section~\ref{sec:hyperstars}, we also compute the profile recording the closure of the points $$\left(\frac{\hom(S^{(k)}_{2k};G)}{(\hom(S^{(k)}_{k};G))^2}, \frac{\hom(S^{(k)}_{3k};G)}{(\hom(S^{(k)}_{k};G))^3}, \ldots, \frac{\hom(S^{(k)}_{\ell k};G)}{(\hom(S^{(k)}_{k};G))^\ell}\right)$$ for $k$-uniform hypergraphs $G$ such that $\hom(S^{(k)}_{k};G)\neq 0$ and where $S^{(k)}_{b}$ is the $k$-uniform hyperstar with $b$ branches that intersect in a single vertex and nowhere else. Note that when $k=2$, $S^{(k)}_b$ is the usual $b$-star graph, i.e., $K_{1,b}$.  Here, because we have a number ratio profile, the usual number profile cannot almost be recovered as discussed in the density ratio profile case above. However, interestingly, we again obtain that this profile is equal to $\Pi_\ell$, the power-sum profile. 

The connection of density profiles of cycles to the power-sum profile comes from considering adjacency matrices. It is well-known that $\operatorname{hom} (C_k,G)=\operatorname{trace} A^k_G=\lambda_1^k+\dots+\lambda_n^k$, where $A_G$ is the adjacency matrix of $G$ and $\lambda_1,\dots,\lambda_n$ are its eigenvalues. A similar connection works for graphs $N_{c,q}$ for $q>2$ where, instead of $A_G$, we use a modified adjacency matrix. Finally, for stars, we have $\operatorname{hom} (S_k,G)=d_1^k+\dots+d_n^k$ where $S_k$ is the star graph with $k$ leaves, and  $d_i$ are the degrees of vertices of $G$. A similar formula holds for hyperstars and uniform hypergraphs;  see Section \ref{sec:hyperstars} for details. It is interesting to note that, while eigenvalues of adjacency matrices and degree sequences of graphs and hypergraphs are special, once we consider graphs of all sizes, we pick up the entirety of the power-sum profile. 


\section{power-sum profile}\label{sec:profiles}

The power-sum profile has intricate and complicated geometry, which is key to proving undecidability results. We recall some geometric results from \cite{abdr}. While these results are not needed for our proofs, they demonstrate the interesting geometry of graphs profiles after applying the ratio map.

First recall the following definition from the introduction. 

\begin{definition}
For $i\in \NN$ and $\xx\in \RR^n$, we let $p_i(\xx)=\sum_{j=1}^n x_j^i$ be the \emph{$i$-th power sum polynomial}. Let $\Pi_{n,\ell}:=\{(p_2(\xx), p_3(\xx), \ldots, p_\ell(\xx)) \mid \xx\in \RR^n_{\geq 0} \textup{ and } p_1(\xx)=1\}$. Let $\Pi_\ell:=\textup{cl}\left(\bigcup_{n\in \NN} \Pi_{n,\ell}\right)$. We call $\Pi_\ell$ the \emph{power-sum profile}.  
\end{definition}

Recall that the \emph{cyclic polytope} $C(n,\ell)$ for $n>\ell\geq 2$ is the convex hull of $n$ points on the real $\ell$-dimensional moment curve $(t, t^2, \ldots, t^\ell)$. In \cite{abdr}, the authors showed that $\Pi_{n,\ell}$ has the combinatorial structure of a cyclic polytope, and that $\Pi_\ell$ has the combinatorial structure of an infinite cyclic polytope. More precisely, they proved the following results.

\begin{theorem}[Theorem 2.4 in \cite{abdr}]
For integers $n\geq \ell$, the boundary of $\Pi_{n,\ell}$ is the set of points $(p_2(\xx), p_3(\xx), \ldots, p_\ell(\xx))$ such that $\xx\in \RR^n_{\geq 0}$ and $p_1(\xx)=1$ that are of the following two types:
\begin{enumerate}
\item $(\underbrace{0,\ldots, 0}_{r_0}, \underbrace{x_1}_{r_1}, \underbrace{x_2, \ldots, x_2}_{r_2}, \ldots, \underbrace{x_{\ell-1}, \ldots, x_{\ell-1}}_{r_{\ell-1}})$ with $r_{2k-1}=1$ and $r_0\geq 0$, $r_{2k}\geq 1$ for all $k$,
\item $(\underbrace{x_1, \ldots, x_1}_{r_1}, \underbrace{x_2}_{r_2}, \ldots, \underbrace{x_{\ell-1}, \ldots, x_{\ell-1}}_{r_{\ell-1}})$ with $r_{2k}=1$ and $r_{2k-1}\geq 1$ for all $k$, 
\end{enumerate}
and $0\leq x_1 \leq x_2 \leq \ldots \leq x_{\ell-1}$. 
\end{theorem}

\begin{theorem}[Theorem 3.1 in \cite{abdr}]
The set $\Pi_{n,\ell}$ has the combinatorial structure of the cyclic polytope $C(n,\ell-1)$, i.e., there is a homeomorphism $\textup{bd } C(n,\ell-1) \rightarrow \textup{bd } \Pi_{n,\ell}$ that is a diffeomorphism when restricted to the relative interior of any face of $\textup{bd }C(n,\ell-1)$. 
\end{theorem}

\begin{figure}
\includegraphics[scale=0.2]{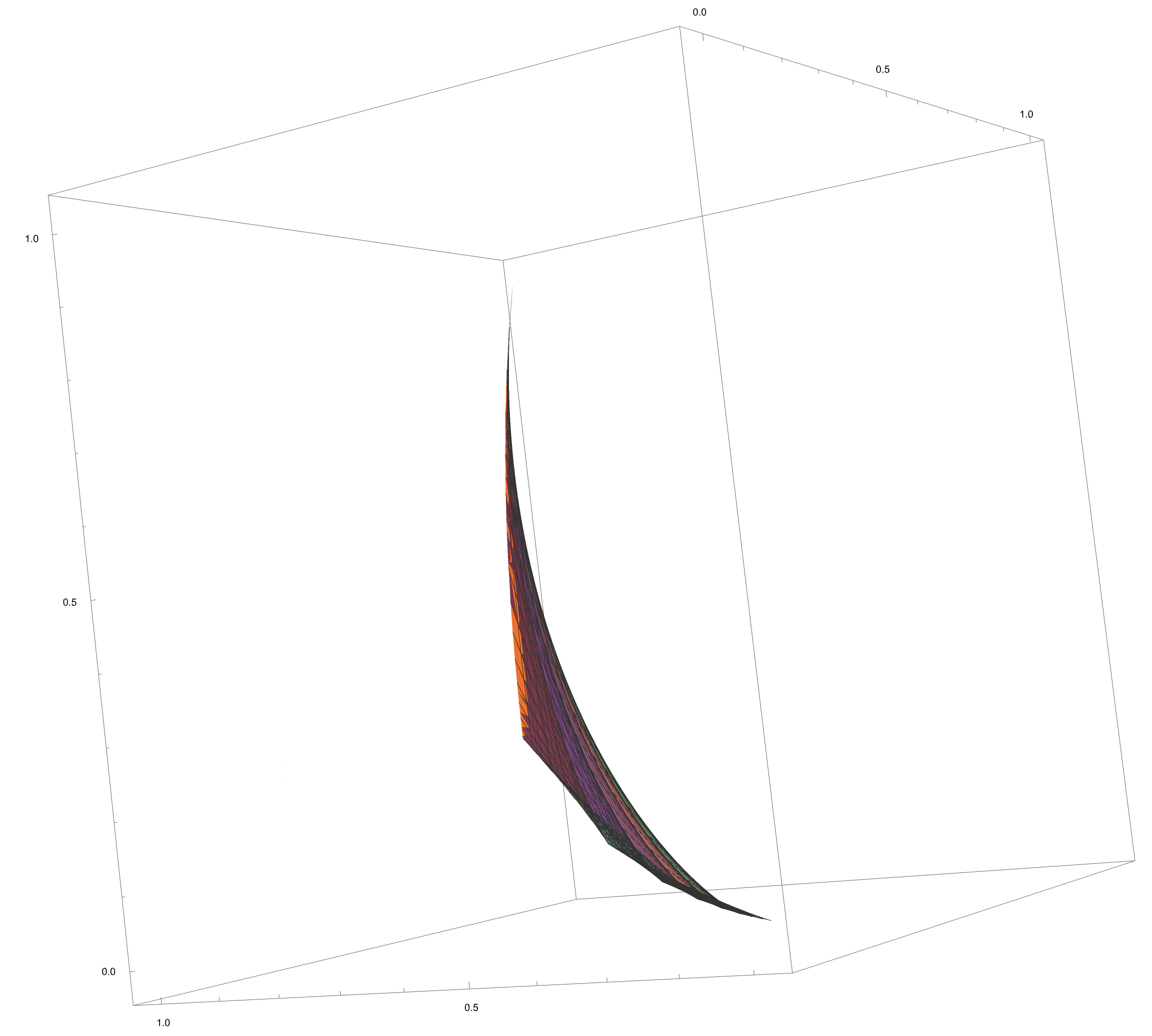}

\caption{The power-sum profile $\Pi_4$.}
\label{fig:psp3d}

\end{figure}

An explicit homeomorphism is given in Section 3 of \cite{abdr}, whereas Section 4 of that paper goes into greater details about $\Pi_\ell$. Loosely speaking, they show that the boundary of $\Pi_\ell$ is a gluing of countably infinitely many patches and each patch is a curved $(\ell-2)$-simplex. Gale's eveness theorem \cite{gale} still holds in this infinite setting, and so the $\ell-1$ vertices of any patch are the union of pairs of consecutive points on the moment curve with any subset of the first point on the moment curve and an accumulation point. Figure~\ref{fig:psp3d} shows $\Pi_4$.


\section{Ratio profile for cycles and necklaces}\label{sec:necklaces}

We build up to our most general profile for ratios of necklaces by first presenting the profile for ratios of cycles to help explain the intuition in the more general setting.

\subsection{Ratio profile for cycles}

Let $C_i$ be the cycle with $i$ vertices and let $$\mathcal{C}_{4\ell}^{\hom}:=\textup{cl}\left(\left\{\left(\frac{\hom(C_8;G)}{(\hom(C_4;G))^2}, \frac{\hom(C_{12};G)}{(\hom(C_4;G))^3}, \ldots, \frac{\hom(C_{4\ell};G)}{(\hom(C_4;G))^\ell}\right) \mid G \textup{ is a graph with } \hom(C_{4};G)\neq 0\right\}\right).$$ We show that this profile is a power-sum profile.

\begin{theorem}\label{thm:cycles}
We have that $\mathcal{C}_{4\ell}^{\hom} = \Pi_\ell$.
\end{theorem}

\begin{proof}
We first show that $\mathcal{C}_{4\ell}^{\hom} \subseteq \Pi_\ell$. Let $A_G$ be the adjacency matrix of a graph $G$ on $n$ vertices. Recall that the entry $(u, v)$ of $A_G^j$ is the number of walks of length $j$ between the vertices $u$ and $v$ of $G$, and in particular, the entry $(v,v)$ of $A^j_G$ is the number of walks of length $j$ starting and ending at some vertex $v$. This means that $\hom(C_{j};G)=\textup{tr}(A_G^j)$ which in turn is equal to $\sum_{i=1}^n \lambda_i^j$ where $\lambda_1, \ldots, \lambda_n$ are the eigenvalues of $A_G$. Therefore, we have that 
$$\mathcal{C}_{4\ell}^{\hom}=\textup{cl}\left(\left\{\frac{\sum_{i} \lambda_i^8}{(\sum_{i} \lambda_i^4)^2}, \frac{\sum_{i} \lambda_i^{12}}{(\sum_{i} \lambda_i^4)^3}, \ldots, \frac{\sum_{i} \lambda_i^{4\ell}}{(\sum_{i} \lambda_i^4)^\ell} \mid G \textup{ is a graph with eigenvalues } \lambda_i \textup{'s s.t.\ } \sum_{i} \lambda_i^4 \neq 0\right\}\right).$$ Fix a graph $G$ on $n$ vertices and where $\lambda_i$'s are the eigenvalues of $A_G$. Letting $x_i=\frac{\lambda_i^4}{\sum_{k=1}^n(\lambda_k^4)}$ for $1\leq i \leq n$ and $x_i=0$ for $i>n$, we have $p_j(\xx)=\frac{\hom(C_{4j};G)}{\hom(C_{4};G)^j}$. Note that $p_1(\xx)=1$ and $x_i\geq 0$ for all $i\in [n]$, so 

$$\mathcal{C}_{4\ell}^{\hom}=\left\{(p_2(\xx), p_3(\xx), \ldots, p_\ell(\xx)) \mid x_i=\frac{\lambda_i^4}{\sum_{k=1}^n(\lambda_k^4)} \forall i\in[n] \textup{ for some graph }G \right\}\subseteq \Pi_\ell.$$ 

We now show that $\mathcal{C}_{4\ell}^{\hom} \supseteq \Pi_\ell$. Consider an arbitrary point $(p_2(\yy), p_3(\yy), \ldots, p_\ell(\yy))\in \Pi_\ell$ where $$\yy := (y_1, y_2, \ldots, y_n, 0, 0, \ldots)$$  for some $n\in \NN$ such that $p_1(\yy) = 1$ and where $y_i\geq 0$ for every $i$. We show that $(p_2(\yy), p_3(\yy), \ldots, p_\ell(\yy))$ is in  $\mathcal{C}_{4\ell}^{\hom}$. 

For $i\in [n]$, let $G_i$ be the graph $K_{\sqrt[4]{y_i} N+1}$. Thus $\sqrt[4]{y_i} N$ appears once among the eigenvalues of $A_{G_i}$ and the remaining $\sqrt[4]{y_i}N$ eigenvalues are $-1$.  Let $G$ be the disjoint union of the $G_i$'s for $i\in [n]$. We thus have that 

\begin{align*} 
\frac{\hom(C_{4j};G)}{\hom(C_4;G)^j} &=\frac{\sum_{i=1}^{n} \left( (\sqrt[4]y_i N)^{4j} + \sqrt[4]{y_i}N (-1)^{4j}\right)}{\left(\sum_{i=1}^{n} \left( (\sqrt[4]y_i N )^{4} + \sqrt[4]{y_i}N (- 1)^{4}\right)\right)^{j}}\\
&=\frac{\sum_{i=1}^{n} (y_i^j N^{4j} + \sqrt[4]{y_i}N)}{\left(\sum_{i=1}^{n} (y_i N^{4} + \sqrt[4]{y_i}N)\right)^j}.\end{align*}

So as $N\rightarrow \infty$, we have that 
$$\frac{\hom(C_{4j};G)}{\hom(C_4;G)^j} \rightarrow \frac{\sum_{i=1}^{n} y_i^j}{\left(\sum_{i=1}^{n} y_i \right)^j} = \frac{p_j(\yy)}{(p_1(\yy))^j}=p_j(\yy),$$ and so $(p_2(\yy), p_3(\yy), \ldots, p_\ell(\yy))\in \mathcal{C}_{4\ell}^{\hom}$.
\end{proof}

\begin{corollary}
Recall that $$\mathcal{C}_{4\ell}=\textup{cl}\left(\left\{\left(\frac{t(C_8;G)}{(t(C_4;G))^2}, \frac{t(C_{12};G)}{(t(C_4;G))^3}, \ldots, \frac{t(C_{4\ell};G)}{(t(C_4;G))^\ell}\right) \mid G \textup{ is a graph s.t.\ } t(C_{4};G)\neq 0\right\}\right).$$ We have that $\mathcal{C}_{4\ell}=\Pi_\ell.$
\end{corollary}

\begin{proof}
First observe that $$\frac{t(C_{4j};G)}{t(C_4;G)^j}=\frac{\frac{\hom(C_{4j};G)}{|V(G)|^{4j}}}{\frac{\hom(C_4;G)^j}{(|V(G)|^4)^j}}=\frac{\hom(C_{4j};G)}{\hom(C_4;G)^j},$$ and so the ratio profile of densities $\mathcal{C}_{4\ell}$ is equal to the ratio profile of numbers $\mathcal{C}_{4\ell}^{\hom}$. The result then follows from Theorem~\ref{thm:cycles}.
\end{proof}

As explained in the introduction, this means that we can recover the usual density profile $$\mathcal{C}^*_{4\ell}:=\textup{cl}\left(\left\{\left(t(C_4,G), t(C_8;G), \ldots, t(C_{4\ell};G)\right) \mid G \textup{ is a graph}\right\}\right)$$ up to one dimension.

\subsection{Ratio profile for necklaces}
We now generalize the previous profile to necklaces. 

\begin{definition}
Given some graph $G$ and some integer $q\geq 3$, the \emph{$q$-ification of $G$} is the graph obtained as follows: for every edge of $G$, add $q-2$ vertices that are all pairwise adjacent and that are all adjacent to the two vertices of the selected edge. In other words, each edge of $G$ gets replaced by a clique of size $q$. More formally, the $q$-ification of $G$ is the graph with vertex set $V=\{v\mid v\in V(G)\}\cup \{(\{u,v\},i)\mid\{u,v\}\in E(G) \textup{ and } i\in [q-2]\}$ and edge set $E=\{\{u,v\}\mid \{u,v\}\in E(G)\}\cup \{\{v,(\{u,v\},i)\mid \{u,v\}\in E(G) \textup{ and } i\in [q-2]\} \cup \{\{(\{u,v\},i), (\{u,v\}, j)\} \mid \{u,v\}\in E(G) \textup{ and } i,j\in [q-2] \textup{ where }  i\neq j\}$.
\end{definition}

Necklaces are defined as the $q$-ification of cycles. 

\begin{definition}
Let $N_{c,q}$ be the $q$-ification of the cycle of length $c$. We call $N_{c,q}$ the \emph{$q$-necklace} of length $c$. For example, Figure~\ref{fig:N54P33} shows a $4$-necklace of length 5. 
\end{definition}

\begin{figure}[h!]
\centering
\includegraphics[scale=0.3]{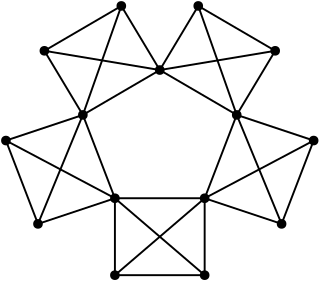} 
\caption{The $4$-necklace of length 5, $N_{5,4}$.}\label{fig:N54P33}
\end{figure}

In the proof of Theorem~\ref{thm:cycles}, we used the adjacency matrix of $G$ to count $\hom(C_c;G)$. We now look at an analogue of the adjacency matrix for cliques so that we can count $\hom(N_{c,q};G)$. 

\begin{definition}\label{def:MWq}
Let $G$ be a graph on $n$ vertices and let $M_{G,q}$ be an $n\times n$ matrix where the entry $(u,v)$ records the number of cliques of size $q$ in $G$ containing edge $\{u,v\}$. Note that $M_{G,2}=A_G$.
\end{definition}

From \cite{BRW}, we  can relate the number of homomorphisms in a necklace to the power sum of the eigenvalues in $M_{G,q}$. Indeed, just as entry $(u,v)$ of $A^j_G$ counts the number of walks of length $j$ in $G$ starting at vertex $u$ and ending at vertex $v$, entry $(u, v)$ of $M_{G,q}^j$ counts the number of $q$-ifications of such walks. In particular, this means that entry $(v,v)$ of $M_{G,q}^j$ counts the number of $q$-ifications of closed walks of length $j$ starting and ending at $v$.

\begin{lemma}\label{lem:countnecklace}
For any graph $G$, we have that $\hom(N_{j,q};G)=\sum_{i} \lambda_i^j$ where the $\lambda_i$'s are the eigenvalues of $M_{G,q}$.
 \end{lemma}

We generalize $\mathcal{C}_{4\ell}^{\hom}$ to $\mathcal{N}_{4\ell,q}^{\hom}$ which we define as $$\textup{cl}\left(\left\{\left(\frac{\hom(N_{8,q};G)}{(\hom(N_{4,q};G))^2}, \frac{\hom(N_{12,q};G)}{(\hom(N_{4,q};G))^3}, \ldots, \frac{\hom(N_{4\ell,q};G)}{(\hom(N_{4,q};G))^\ell}\right) \mid G \textup{ is a graph with } \hom(N_{4,q};G)\neq 0\right\}\right).$$ Note that $\mathcal{N}_{4\ell,2}^{\hom}=\mathcal{C}_{4\ell}^{\hom}$. Using more or less the same proof as for Theorem~\ref{thm:cycles} but with our $q$-generalization of the adjacency matrix, we show that this profile is a power-sum profile.

\begin{theorem}\label{thm:necklacessingleq}
Fix $q\geq 2$. We have that $\mathcal{N}_{4\ell, q}^{\hom} = \Pi_\ell$.
\end{theorem}

\begin{proof}

We first show that $\mathcal{N}_{4\ell,q}^{\hom} \subseteq \Pi_\ell$.  From Lemma~\ref{lem:countnecklace}, since $\hom(N_{4j,q};G)=\sum_{i} \lambda_i^{4j}$ where the $\lambda_i$'s are the eigenvalues of $M_{G,q}$, we have $$\mathcal{N}_{4\ell,q}^{\hom}=\textup{cl}\left(\left\{\frac{\sum_{i} \lambda_i^8}{(\sum_{i} \lambda_i^4)^2}, \frac{\sum_{i} \lambda_i^{12}}{(\sum_{i} \lambda_i^4)^3}, \ldots, \frac{\sum_{i} \lambda_i^{4\ell}}{(\sum_{i} \lambda_i^4)^\ell} \mid G \textup{ is a graph s.t.\ } \sum_{i} \lambda_i^4 \neq 0\right\}\right).$$ Fix a graph $G$ on $n$ vertices and where $\lambda_i$'s are the eigenvalues of $M_{G,q}$. Letting $x_i=\frac{\lambda_i^4}{\sum_{k=1}^n(\lambda_k^4)}$ for $1\leq i \leq n$ and $x_i=0$ for $i>n$, we have $p_j(\xx)=\frac{\hom(N_{4j,q};G)}{\hom(N_{4,q};G)^j}$. Note that $p_1(\xx)=1$ and $x_i\geq 0$ for all $i$, so 

$$\mathcal{N}_{4\ell,q}^{\hom}=\left\{(p_2(\xx), p_3(\xx), \ldots, p_\ell(\xx)) \mid x_i=\frac{\lambda_i^4}{\sum_{k=1}^n(\lambda_k^4)} \forall i\in[n] \textup{ for some graph }G \right\}\subseteq \Pi_\ell.$$

We now show that $\mathcal{N}_{4\ell,q}^{\hom} \supseteq \Pi_\ell$. Consider an arbitrary point $(p_2(\yy), p_3(\yy), \ldots, p_\ell(\yy))\in \Pi_\ell$ where $$\yy := (y_1, y_2, \ldots, y_n, 0, 0, \ldots)$$ for some $n\in \NN$ such that $p_1(\yy)= 1$ and where $y_i\geq 0$ for every $i$. We show that $(p_2(\yy), p_3(\yy), \ldots, p_\ell(\yy))$ is in $\mathcal{N}_{4\ell, q}^{\hom}$ for every $q\geq 3$. We have already settled the case $q=2$ in Theorem~\ref{thm:cycles}. 

For $i\in[n]$, let $G_i$ be the graph $K_{z_iN+1}$ for some large $N$ where $z_i=y_i^{\frac{1}{4(q-1)}}$. Assuming that $z_iN+1 \geq q$, then each edge of $G_i$ is contained in exactly $\binom{z_iN-1}{q-2}$ cliques of size $q$ within $K_{z_iN+1}$. So $M_{G_i,q}$ has the form $\binom{z_iN-1}{q-2}(P-I)$ where $P$ is the all-ones matrix and $I$ is the identity matrix, both of size $((z_iN+1)\times (z_iN+1))$. Thus $z_iN \binom{z_iN-1}{q-2}$ appears once among the eigenvalues of $M_{G_i,q}$ and the remaining $z_iN$ eigenvalues are $-\binom{z_iN-1}{q-2}$. Let $G$ be the disjoint union of the $G_i$'s for $i\in [n]$. We then have

\begin{align*} 
\frac{\hom(N_{4j,q};G)}{\hom(N_4;G)^j} &=\frac{\sum_{i=1}^{n}  \left( \left(z_iN \binom{z_i N -1}{q-2} \right)^{4j} +z_iN \left( - \binom{z_i N-1}{q-2} \right)^{4j}\right)}{\left(\sum_{i=1}^{n}  \left( \left(z_i N \binom{z_i N -1}{q-2} \right)^{4} + z_i N \left( - \binom{z_i N-1}{q-2} \right)^{4}\right)\right)^{j}}\\
&=\frac{\sum_{i=1}^{n}  (z_i^{4j} N^{4j} + z_i N) \binom{z_iN-1}{q-2}^{4j}}{\left(\sum_{i=1}^{n}  (z_i^4 N^{4} + z_iN) \binom{z_i N-1}{q-2}^{4}\right)^j}\\
&=\frac{\sum_{i=1}^{n} (z_i^{4j} N^{4j} + z_i N) \left(\frac{z_i^{q-2}N^{q-2}}{(q-2)!}+O(N^{q-3})\right)^{4j}}{\left(\sum_{i=1}^{n} (z_i^4 N^{4} + z_iN)\left(\frac{z_i^{q-2}N^{q-2}}{(q-2)!}+O(N^{q-3})\right)^{4}\right)^j}.
\end{align*}

Note that the largest degree term in $N$ in each summand of the numerator is $$z_i^{4j}N^{4j} \frac{z_i^{(q-2)4j}}{((q-2)!)^{4j}}N^{4j(q-2)}=\frac{z_i^{(q-1)4j}N^{4j(q-1)}}{((q-2)!)^{4j}}=\frac{y_i^j N^{4j(q-1)}}{((q-2)!)^{4j}}.$$ Similarly, the largest degree term in each summand of the denominator (before taking the $j$th power) is $$z_i^{4}N^{4} \frac{z_i^{(q-2)4}}{((q-2)!)^{4}}N^{4(q-2)}=\frac{z_i^{(q-1)4}N^{4(q-1)}}{((q-2)!)^{4}}=\frac{y_i N^{4(q-1)}}{((q-2)!)^{4}}.$$ So as $N\rightarrow \infty$, we have that 
\begin{align*} 
\frac{\hom(N_{4j,q};G)}{\hom(N_4;G)^j} & \rightarrow \frac{\sum_{i=1}^{n}\frac{y_i^j N^{4j(q-1)}}{((q-2)!)^{4j}}}{\left(\sum_{i=1}^{n}\frac{y_i N^{4(q-1)}}{((q-2)!)^{4}}  \right)^j}= \frac{\frac{N^{4j(q-1)}}{((q-2)!)^{4j}}\sum_{i=1}^{n}  y_i^j }{\left(\frac{N^{4(q-1)}}{((q-2)!)^{4}} \right)^j \left(\sum_{i=1}^{n}y_i  \right)^j}=\frac{\sum_{i=1}^{n}  y_i^j}{\left(\sum_{i=1}^{n} y_i \right)^j} = \frac{p_j(\yy)}{(p_1(\yy))^j}=p_j(\yy),
\end{align*}
and so $(p_2(\yy), p_3(\yy), \ldots, p_\ell(\yy))\in \mathcal{N}_{4\ell,q}^{\hom}$.
\end{proof}

We now finally generalize $\mathcal{N}_{4\ell,q}^{\hom}$ to a larger family of ratios going over different $q$'s. Let $\mathcal{N}_{4\ell,\leq r}^{\hom}$ be the closure of the points of the form \begin{small}$$\left(\frac{\hom(N_{8,2};G)}{(\hom(N_{4,2};G))^2}, \frac{\hom(N_{12,2};G)}{(\hom(N_{4,2};G))^3}, \ldots, \frac{\hom(N_{4\ell,2};G)}{(\hom(N_{4,2};G)}, \ldots, \frac{\hom(N_{8,r};G)}{(\hom(N_{4,r};G))^2}, \frac{\hom(N_{12,r};G)}{(\hom(N_{4,r};G))^3}, \ldots, \frac{\hom(N_{4\ell,r};G)}{(\hom(N_{4,r};G))^\ell}\right)$$\end{small} where $G$  is a graph such that $\hom(N_{4,q};G)\neq 0$ for any $2 \leq q \leq r$.

The argument will proceed as before, however, we need a more complicated construction that we now introduce. 

\begin{definition}
Let $A(k,2)$ be a triangle-free graph with $n$ vertices that is $d$-regular and where the second largest eigenvalue (in absolute value) of the adjacency matrix of $A(k,2)$ is $\lambda$, where $n=\Theta(2^{3k})$, $d=\Theta(2^{2k})$, and $\lambda=\Theta(2^k)$. Here the constants in $\Theta$ are absolute constants. For an integer $q\geq 3$, let $A(k,q)$ be the  $q$-ification of $A(k,2)$.

The graphs $A(k,2)$ are known as $(n,d,\lambda)$-graphs in the literature, and were first constructed by Alon in \cite{Alon}. 
 Since $A(k,2)$ has $\frac{nd}{2}=\Theta(2^{5k-1})$ edges and since the $q$-ification $A(k,q)$ introduces $q-2$ vertices for every edge of $A(k,2)$, $A(k,q)$ has $n+(q-2)\frac{dn}{2}=\Theta(2^{3k}+(q-2)2^{5k-1})$ vertices. The original vertices coming from $A(k,2)$ have degree $d(q-1)$, and the new vertices have degree $q-1$. Further, the original vertices are each contained in $d$ $q$-cliques, whereas the new vertices are contained in exactly one $q$-clique. 
\end{definition}

From Corollary 2.19 and Lemma 2.20 of \cite{BRW}, the following statement can be made about the spectrum of $M_{A(k,q),p}$

\begin{lemma}\label{lem:brwspectrum}
Fix $q\geq p \geq 2$. The top eigenvalue of $M_{A(k,q),p}$ is $\Theta(2^{2k})$, at most $O(2^{3k})$ eigenvalues are $O(2^{k})$, and the rest are $O(1)$. 
\end{lemma}

We now informally explain the strategy of the proof of Theorem~\ref{thm:necklaces}. A key point is that $\hom(N_{4j,p};A(k,q))$ is 0 if $q<p$ since $N_{4j,p}$ contains bigger cliques than $A(k,q)$. This allows us to independently realize different points in $\Pi_\ell$ as points in $\mathcal{N}_{4\ell,p}^{\hom}$ for different $2\leq p\leq q$. Start with the $p=q$, and realize the desired point in $\Pi_\ell$ with disjoint copies of $A(k_i,q)$'s for different $k_i$'s. Now go to $p=q-1$. Again we realize the desired point in $\Pi_\ell$ with disjoint copies of $A(k_i,q)$'s for different $k_i$'s, but we choose $k_i$'s that ensure the ratio obtained for $p=q$. And so on for smaller values of $p$. 

\begin{theorem}\label{thm:necklaces}
We have that $\mathcal{N}_{4\ell, \leq r}^{\hom} = \Pi_\ell^{r-1}$.
\end{theorem}

\begin{proof}

By Theorem~\ref{thm:necklacessingleq}, we have that $\mathcal{N}_{4\ell,\leq r}^{\hom} \subseteq   \Pi_\ell^{r-1}$.

To prove that $\mathcal{N}_{4\ell,\leq r}^{\hom} \supseteq  \Pi_\ell^{r-1}$, we show that every point in $ \Pi_\ell^{r-1}$ can be realized as  \begin{small}$$\left(\frac{\hom(N_{8,2};G)}{(\hom(N_{4,2};G))^2}, \frac{\hom(N_{12,2};G)}{(\hom(N_{4,2};G))^3}, \ldots, \frac{\hom(N_{4\ell,2};G)}{(\hom(N_{4,2};G)}, \ldots, \frac{\hom(N_{8,r};G)}{(\hom(N_{4,r};G))^2}, \frac{\hom(N_{12,r};G)}{(\hom(N_{4,r};G))^3}, \ldots, \frac{\hom(N_{4\ell,r};G)}{(\hom(N_{4,r};G))^\ell}\right)$$\end{small} for some sequence of graphs $G$.

Consider an arbitrary point in $ \Pi_\ell^{r-1}$, say
$$\left(p_2(\yy^{(2)}), p_3(\yy^{(2)}), \ldots, p_{\ell}(\yy^{(2)}), p_2(\yy^{(3)}), p_3(\yy^{(3)}), \ldots, p_{\ell}(\yy^{(3)})\ldots, p_2(\yy^{(r)}), p_3(\yy^{(r)}), \ldots, p_{\ell}(\yy^{(r)})\right)$$
where $\yy^{(q)}=(y_1^{(q)}, y_2^{(q)}, \ldots, y_{n_q}^{(q)}, 0, 0, \ldots)$ for some $n_q\in \NN$ such that $p_1(\yy^{(q)})=1$ and $y_i^{(q)}\geq 0$ for all $0\leq i \leq n_q$ and $2\leq q \leq r$. Thus, $p_j(\yy^{(q)})=\sum_{i=1}^{n_q} {y_i^{(q)}}^j$ for $2 \leq q \leq r$ and $2\leq j\leq \ell$.

For $2\leq q \leq r$ and $1\leq i \leq n_q$, let $G_i^{(q)}$ be the graph $A(k_i^{(q)},q)$ where $k_i^{(q)}= \frac{1}{8} \log_2 (y_i^{(q)}N^{r-q})$ for some large positive $N$. Fix $2\leq q \leq r$ and $2\leq p\leq r$, and let $\mu_{i, q,p,t}$ be the $t$th largest eigenvalue of $M_{A(k_i^{(q)},q),p}$. Observe that from  Lemma~\ref{lem:countnecklace} and Lemma~\ref{lem:brwspectrum}, $\hom(N_{4j,p};A(k_i^{(q)},q))$ is equal to $$\sum_t \mu_{i,q,p,t}^{4j} = \Theta(2^{2k_i^{(q)}})^{4j}+O(2^{3k_i^{(q)}})\cdot O(2^{k_i^{(q)}})^{4j}+O(2^{5k_i^{(q)}-1})\cdot O(1)^{4j} = (1+o(1))\Theta(2^{2k_i^{(q)}})^{4j}$$ for any $j\geq 1$.  Let $G$ be the disjoint union of the $G_i^{(q)}$ for $i\in [n_q]$ and $2\leq q \leq r$. We then have that 

\begin{align*}
\frac{\hom(N_{4j,p};G)}{\hom(N_{4,p};G)^j}&= \frac{\sum_{q=2}^r \sum_{i=1}^{n_q}  \hom(N_{4j,p};A(k_i^{(q)},q))}{\left(\sum_{q=2}^r \sum_{i=1}^{n_q} \hom(N_{4,p};A(k_i^{(q)},q))\right)^j}\\
&= \frac{\sum_{q=p}^r \sum_{i=1}^{n_q}  \hom(N_{4j,p};A(k_i^{(q)},q))}{\left(\sum_{q=p}^r \sum_{i=1}^{n_q}  \hom(N_{4,p};A(k_i^{(q)},q))\right)^j}\\
&= \frac{\sum_{q=p}^r \sum_{i=1}^{n_q} (1+o(1)){\Theta(2^{2k_i^{(q)}})^{4j}}}{\left(\sum_{q=p}^r \sum_{i=1}^{n_q}  (1+o(1)){\Theta(2^{2k_i^{(q)}})}^{4}\right)^j}\\
&= \frac{\sum_{q=p}^r \sum_{i=1}^{n_q} (1+o(1))(y_i^{(q)}N^{r-q})^j}{\left(\sum_{q=p}^r \sum_{i=1}^{n_q} (1+o(1))(y_i^{(q)}N^{r-q})\right)^j}\\
\end{align*}
where the second line follows from the fact that $\hom(N_{4j,p};A(k_i^{(q)},q))=0$ if $q<p$, where the third line follows from the observation above, and where the fourth line follows follows from the definition of $k_i^{(q)}$. 

The degree of the numerator and denominator is both $j(r-p)$. So as $N\rightarrow \infty$, we get that 

$$\frac{\hom(N_{4j,p};G)}{\hom(N_{4,p};G)^j}\rightarrow \frac{\sum_{i=1}^{n_p} {y_i^{(p)}}^j}{\left( \sum_{i=1}^{n_p} y_i^{(p)}\right)^j}=\frac{p_j(\yy^{(p)})}{(p_1(\yy^{(p)}))^j}=p_j(\yy^{(p)})$$
as desired, and so $$\left(p_2(\yy^{(2)}), p_3(\yy^{(2)}), \ldots, p_{\ell}(\yy^{(2)}), p_2(\yy^{(3)}), p_3(\yy^{(3)}), \ldots, p_{\ell}(\yy^{(3)})\ldots, p_2(\yy^{(r)}), p_3(\yy^{(r)}), \ldots, p_{\ell}(\yy^{(r)})\right)\in \mathcal{N}_{4\ell,\leq r}^{\hom}.$$ 
\end{proof}

\begin{corollary}\label{cor:density}
Recall that $\mathcal{N}_{4\ell, \leq r}$ is the closure of the points of the form \begin{small}$$\left(\frac{t(N_{8,2};G)}{(t(N_{4,2};G))^2}, \frac{t(N_{12,2};G)}{(t(N_{4,2};G))^3}, \ldots, \frac{t(N_{4\ell,2};G)}{(t(N_{4,2};G)}, \ldots, \frac{t(N_{8,r};G)}{(t(N_{4,r};G))^2}, \frac{t(N_{12,r};G)}{(t(N_{4,r};G))^3}, \ldots, \frac{t(N_{4\ell,r};G)}{(t(N_{4,r};G))^\ell}\right)$$\end{small} where $G$  is a graph such that $t(N_{4,q};G)\neq 0$ for any $2 \leq q \leq r$. Then $\mathcal{N}_{4\ell, \leq r}$ is equal to $\Pi_\ell^{r-1}$.
\end{corollary}

\begin{proof}
We have that $$\frac{t(N_{4j,q};G)}{t(N_{4,q};G)^j}=\frac{\frac{\hom(N_{4j,q};G)}{|V(G)|^{4j(q-2)}}}{\frac{\hom(N_{4,q};G)^j}{(|V(G)|^{4(q-2)})^j}}=\frac{\hom(N_{4j,q};G)}{\hom(N_{4,q};G)^j}.$$ So the density ratio profile $\mathcal{N}_{4\ell, \leq r}$ is equal to the number ratio profile $\mathcal{N}_{4\ell, \leq r}^{\hom}$, and the result holds by Theorem~\ref{thm:necklaces}.
\end{proof}

\begin{remark}\label{rem:lossofdimensionsgeneral}
As mentioned in the introduction, since $\mathcal{C}_{4\ell}$, $\mathcal{N}_{4\ell,q}$ and $\mathcal{N}_{4\ell,\leq r}$ are all \emph{density ratio} profiles, they contain high-dimensional information about the corresponding density profiles. We now discuss what it would take to recover the  density profile from the ratio profile in the most general case  $\mathcal{N}_{4\ell,\leq r}$ (the cases  $\mathcal{C}_{4\ell}$ and $\mathcal{N}_{4\ell,q}$ are discussed in the introduction). Recall that we let $$\mathcal{N}^*_{4\ell,\leq r}:=\textup{cl}\left(\left\{\left(t(N_{4,2};G), t(N_{8,2};G), \ldots,(N_{4\ell,2};G), \ldots t(N_{4,r};G), t(N_{8,r};G), \ldots,(N_{4\ell,r};G)\right) \mid G \textup{ is a graph}\right\}\right)$$ be the usual density profile for all $(4i,q)$-necklaces for $1\leq i \leq \ell$ and $2\leq q \leq r$. To obtain $\mathcal{N}^*_{4\ell, \leq r}$ from $\mathcal{N}_{4\ell, \leq r}$, we need to compute the fiber over any point $(\mathbf{a}_2, \ldots, \mathbf{a}_r)\in \Pi_\ell^{r-1}$. 

Observe that for any point $\mathbf{b}=(b_{4,2},b_{8,2},\ldots, b_{4\ell,2}, \ldots, b_{4,r}, b_{8,r}, \ldots, b_{4\ell,r}) \in \mathcal{N}^*_{4\ell, \leq r}$, we know that for any $0\leq t \leq 1$, the point $\mathbf{b}'$ where $b'_{4i,q}=b_{4i,q}t^{4i(q-1)}$ for $1\leq i\leq \ell$ and $2 \leq q \leq r$ is also in $\mathcal{N}^*_{4\ell, \leq r}$. Indeed, if $G_n$ is the sequence of graphs that realizes $\mathbf{b}$, then $\mathbf{b}'$ can be realized with $G_n$ with some added isolated vertices. Moreover, note that $\mathbf{b}$ and $\mathbf{b}'$ project down to the same $(\mathbf{a}_2, \ldots, \mathbf{a}_r)\in \Pi_\ell^{r-1}$ since $\frac{b_{4i,q}}{b_{4,q}^i} = \frac{b_{4i,q}t^{4i(q-1)}}{(b_{4,q}t^{4(q-1)})^i}=\frac{b'_{4i,q}}{(b'_{4,q})^i}$ for every $2\leq i\leq \ell$, $2 \leq q \leq r$ (if $t>0$). 

So one can partition the points of $\mathcal{N}^*_{4\ell, \leq r}$ so that $\mathbf{b}$ and $\mathbf{b}'$ are in the same part if there exists $t>0$ such that $b'_{4i,q}=b_{4i,q}t^{4i(q-1)}$ for $1\leq i\leq \ell$ and $2 \leq q \leq r$. Then all points in one part will project down to the same point in $\Pi_\ell^{r-1}$. 

We thus recover $\mathcal{N}^*_{4\ell, \leq r}$ up to $r-1$ dimensions, even though the dimension of the profile itself can be arbitrarily large. We emphasize that from \cite{BRW}, we know that determining the validity of a polynomial inequality in the densities of the $N_{4i,q}$'s is an undecidable problem.

\end{remark}


\section{Ratio profile for hyperstars}\label{sec:hyperstars}

Let $S^{(k)}_{b}$ be the $k$-uniform hyperstar with $b$ branches all of which intersect in one vertex, called center, and nowhere else. For example, $S^{(2)}_{b}=K_{1,b}$, i.e., the usual $b$-star graph. Let $\mathcal{S}_{\ell,k}$ be the closure of the points of the form $$\left(\frac{\hom(S^{(k)}_{2k};G)}{(\hom(S^{(k)}_{k};G))^2}, \frac{\hom(S^{(k)}_{3k};G)}{(\hom(S^{(k)}_{k};G))^3}, \ldots, \frac{\hom(S^{(k)}_{\ell k};G)}{(\hom(S^{(k)}_{k};G))^\ell}\right)$$
for all $k$-uniform hypergraphs $G$ such that  $\hom(S^{(k)}_{k};G)\neq 0$. 

\begin{theorem}\label{thm:stars}
We have that $\mathcal{S}_{\ell,k} = \Pi_\ell$.
\end{theorem}

\begin{proof}

We first show that $\mathcal{S}_{\ell,k} \subseteq \Pi_\ell$. For any $v \in V(G)$, let $d(v)$ be the number of $k$-hyperedges in $G$ that contain vertex $v$. Note that $$\hom(S^{(k)}_{b};G)=\sum_{v \in V(G)} \big((k-1)!d(v)\big)^b.$$ Therefore, we have that $\mathcal{S}_{\ell,k}$ is the closure of the points 
$$\left(\frac{\sum_{v \in V(G)} \big((k-1)!d(v)\big)^{2k}}{\left(\sum_{v \in V(G)} \big((k-1)!d(v)\big)^k\right)^2}, \frac{\sum_{v \in V(G)} \big((k-1)!d(v)\big)^{3k}}{\left(\sum_{v \in V(G)} \big((k-1)!d(v)\big)^k\right)^3}, \ldots, \frac{\sum_{v \in V(G)} \big((k-1)!d(v)\big)^{\ell k}}{\left(\sum_{v \in V(G)} \big((k-1)!d(v)\big)^k\right)^\ell}\right)$$ 
for all $k$-uniform hypergraphs $G$ such that  $\sum_{v \in V(G)} ((k-1)!d(v))^k \neq 0$.

 Fix a graph $G$ on $n$ vertices, say $V(G):=\{v_1, \ldots, v_n\}$. Letting $$x_i=\frac{\big((k-1)! d(v_i)\big)^k}{\sum_{a=1}^n\left(\big((k-1)!d(v_a)\big)^k\right)}$$ for $1\leq i \leq n$ and $x_i=0$ for $i>n$, we have $$p_j(\xx)=\sum_{i=1}^n \left(\frac{\big((k-1)! d(v_i)\big)^k}{\sum_{a=1}^n\left(\big((k-1)!d(v_a)\big)^k\right)} \right)^j=\frac{\hom(S^{(k)}_{jk};G)}{\hom(S^{(k)}_k;G)^j}.$$ Note that $p_1(\xx)=1$ and $x_i\geq 0$ for all $i$, so 

$$\mathcal{S}_{\ell, k}=\left\{(p_2(\xx), p_3(\xx), \ldots, p_\ell(\xx)) \mid x_i=\frac{\big((k-1)! d(v_i)\big)^k}{\sum_{a=1}^n\left(\big((k-1)!d(v_a)\big)^k\right)} \forall i\in[n]\right\}\subseteq \Pi_\ell.$$

We now show that $\mathcal{S}_{\ell, k} \supseteq \Pi_\ell$. Consider an arbitrary point $(p_2(\yy), p_3(\yy), \ldots, p_\ell(\yy))\in \Pi_\ell$ where $$\yy := (y_1, y_2, \ldots, y_n, 0, 0, \ldots)$$ for some $n\in \NN$ such that $p_1(\yy)= 1$ and where $y_i\geq 0$ for every $i$. We show that $(p_2(\yy), p_3(\yy), \ldots, p_\ell(\yy))$ is in $\mathcal{S}_{\ell,k}$. 

For $i\in [n]$, let $G_i$ be the graph $S^{(k)}_{b_i}$ where $b_i=\frac{(y_i)^{\frac{1}{k}}N}{(k-1)!}$ for some large $N$. Observe that $S^{(k)}_{b_i}$ has one vertex with degree $b_i$ and $b_i(k-1)$ vertices with degree 1. Let $G$ be the disjoint union of the $G_i$'s for $i\in [n]$. Then we have

$$\frac{\hom(S^{(k)}_{kj};G)}{\hom(S^{(k)}_k;G)^j} = \frac{\sum_{i=1}^{n} \left(\left((k-1)!\frac{(y_i)^{\frac{1}{k}}N}{(k-1)!}\right)^{jk}+\frac{(y_i)^{\frac{1}{k}}N}{(k-1)!}\cdot (k-1)\cdot \big((k-1)!\cdot 1\big)^{jk}\right)}{\left(\sum_{i=1}^{n} \left(\left((k-1)!\frac{(y_i)^{\frac{1}{k}}N}{(k-1)!}\right)^{k}+\frac{(y_i)^{\frac{1}{k}}N}{(k-1)!}\cdot (k-1)\cdot \big((k-1)!\cdot 1\big)^{k}\right)\right)^j}$$
which, as $N\rightarrow \infty$, goes to

$$\frac{\sum_{i=1}^{n} y_i^j}{\left(\sum_{i=1}^{n} y_i\right)^j} = \frac{p_j(\yy)}{p_1(\yy)^j} = p_j(\yy),$$ and so $(p_2(\yy), p_3(\yy), \ldots, p_\ell(\yy))$ is in  $\mathcal{S}_{\ell,k}$.

\end{proof}

 \bibliographystyle{alpha}
\bibliography{powersumprofilereferences}

\begin{thebibliography}{ABDR23}

\bibitem[ABDR23]{abdr}
J.~Acevedo, G.~Blekherman, S.~Debus, and C.~Riener.
\newblock The wonderful geometry of the {V}andermonde map.
\newblock {\em arXiv:2303.09512}, 2023.

\bibitem[Alo94]{Alon}
N.~Alon.
\newblock Explicit {R}amsey graphs and orthonormal labelings.
\newblock {\em Electron. J. Combin.}, 1:R12, 1994.

\bibitem[BRW22]{BRW}
G.~Blekherman, A.~Raymond, and F.~Wei.
\newblock Undecidability of polynomial inequalities in weighted graph
  homomorphism densities.
\newblock {\em arXiv:2207.12378}, 2022.

\bibitem[Gal63]{gale}
D.~Gale.
\newblock Neighborly and cyclic polytopes.
\newblock {\em Proc. Sympos. Pure Math.}, 7(7):225--232, 1963.

\bibitem[Goo59]{goodman}
A.W. Goodman.
\newblock On sets of acquaintances and strangers at any party.
\newblock {\em American Mathematical Monthly}, 66:778--783, 1959.

\bibitem[HN11]{HN11}
H.~Hatami and S.~Norin.
\newblock Undecidability of linear inequalities in graph homomorphism
  densities.
\newblock {\em J. Amer. Math. Soc.}, 24(2):547--565, 2011.

\bibitem[IR95]{IR}
Y.~Ioannidis and R.~Ramakrishnan.
\newblock Containment of conjunctive queries: Beyond relations as sets.
\newblock {\em ACM Transactions on Database Systems (TODS)}, 20(3):288--324,
  1995.

\bibitem[Man07]{mantel}
W.~Mantel.
\newblock Problem 28.
\newblock {\em Wiskundige Opgaven}, 10(60-61):320, 1907.

\bibitem[Nik11]{Nikiforov}
V.~Nikiforov.
\newblock The number of cliques in graphs of given order and size.
\newblock {\em Trans. Amer. Math. Soc.}, 363(3):1599--1618, 2011.

\bibitem[Raz08]{RazTriangle}
A.~Razborov.
\newblock On the {M}inimal {D}ensity of {T}riangles in {G}raphs.
\newblock {\em Combinatorics, Probability and Computing}, 17(4):603--618, 2008.

\bibitem[Rei16]{Reiher}
C.~Reiher.
\newblock The clique density problem.
\newblock {\em Annals of Mathematics}, 184:683--707, 2016.

\bibitem[Sid93]{Sid93}
A.F. Sidorenko.
\newblock A correlation inequality for bipartite graphs.
\newblock {\em Graphs and Combinatorics}, 9:201--204, 1993.

\end{thebibliography}

\end{document}